\documentclass[10pt, a4paper]{amsart}
\usepackage{amssymb,amsmath,latexsym,amscd,amsfonts,amsthm,mathrsfs}
\usepackage{enumerate, graphicx}
\usepackage[utf8]{inputenc}

\newcommand{\re}{\text{\rm Re\,}}
\newcommand{\im}{\text{\rm Im\,}}

\newcommand{\bd}{{\mathbb{D}}}
\newcommand{\bn}{{\mathbb{N}}}
\newcommand{\br}{{\mathbb{R}}}

\newcommand{\bc}{{\mathbb{C}}}

\newcommand{\bt}{{\mathbb{T}}}

\newcommand{\ca}{{\mathcal{A}}}
\newcommand{\css}{{\mathcal{S}}}

\renewcommand{\a}{\alpha}
\renewcommand{\b}{\beta}
\renewcommand{\l}{\lambda}
\newcommand{\s}{\sigma}

\renewcommand{\ss}{\Sigma}

\newcommand{\p}{\varphi}
\renewcommand{\t}{\tau}

\renewcommand{\o}{\omega}
\newcommand{\oo}{\Omega}
\newcommand{\g}{\gamma}
\newcommand{\gga}{\Gamma}
\newcommand{\ep}{\varepsilon}
\newcommand{\z}{\zeta}

\newcommand{\nt}{\noindent}

\newcommand{\bsl}{\backslash}

\newcommand{\pt}{\partial}

\newcommand{\lp}{\left(}
\newcommand{\rp}{\right)}

\DeclareMathOperator{\dist}{{\it d}}

\newcommand{\cn}{\mathcal{N}}

\newcommand{\Sc}{\mathcal{S}}

\allowdisplaybreaks
\numberwithin{equation}{section}

\newtheorem{theorem}{Theorem}[section]
\newtheorem*{theoremm}{Theorem A}
\newtheorem{lemma}[theorem]{Lemma}
\newtheorem{corollary}[theorem]{Corollary}
\newtheorem{proposition}[theorem]{Proposition}

\newtheorem{remark}[theorem]{Remark}
\newtheorem*{remarkk}{Remark}

\begin{document}

\title[Generalized Blaschke-type conditions]
{On zeros of analytic functions satisfying non-radial growth conditions}

\author{A. Borichev}
\address{Aix Marseille Universit\'e\\
CNRS\\ Centrale Marseille\\ I2M\\ 39 rue F.~Joliot-Curie, 13453 Marseille\\ France}
\email{alexander.borichev@math.cnrs.fr}

\author{L. Golinskii}
\address{Mathematics Division, ILTPE, 47 Science ave.,  61103 Kharkov, Ukraine}
\email{golinskii@ilt.kharkov.ua}

\author{S. Kupin}
\address{IMB, Universit\'e de Bordeaux, 351 ave. de la Lib\'eration, 33405 Talence Cedex, France}
\email{skupin@math.u-bordeaux1.fr}


\begin{abstract}
Extending the results of Borichev--Golinskii--Kupin [2009], we obtain refined Blaschke-type necessary conditions
on the zero distribution of analytic functions on the unit disk and on the complex plane with a cut along the positive semi-axis
satisfying some non-radial growth restrictions.
\end{abstract}

\maketitle

\begin{center}
{\it To Peter Yuditskii
on occasion of his 60-th anniversary}
\end{center}

\vspace{0.5cm}
\section*{Introduction and main results}\label{s0}

The study of relations between the zero distribution of an analytic function and its growth is likely to be one of the most basic
problems of complex analysis. We have no intention to review a vast literature on it, but just give se\-ve\-ral references related to the points of our
interest. Perhaps, the first results in this direction were obtained in the second half of 19-th century by Hadamard, Borel, Wejerstrass and others,
see Levin \cite[Ch. 2]{le} for a modern presentation. These results completely described the behavior of zeros of an entire function of finite type.
Later, Blaschke \cite{bla}, Nevanlinna \cite{ne} and Smirnov \cite{smi} described the zero sets of functions from the Hardy spaces
$H^p(\bd),\ p>0$, or, more generally, the Nevanlinna class $\cn(\bd)$. Here, as usual, $\bd=\{|z|<1\}$. Namely, for
$f\in\cn(\bd),\  f\not\equiv0$, one has
\begin{equation}\label{ee1}
\sum_{\zeta\in Z(f)} (1-|\zeta|)\le\sup_{0\le r<1} \frac1{2\pi}\,\int^{2\pi}_0\log^+|f(re^{i\theta})|\, d\theta - \log|f(0)|,
\end{equation}
where $Z(f)$ stands for the zero set of $f$ counting multiplicities. Hence, a discrete subset $Z(f)$ of the unit disk is a zero set of a function from
$H^p(\bd)$ (or $\cn(\bd)$) if and only if the series at the LHS of \eqref{ee1} converges. This condition is usually called the ``Blaschke condition''
after \cite{bla}.

Let $\ca(\bd)$ be the set of analytic functions on the unit disk. An argument similar to the proof of \eqref{ee1}, shows that
if $f\in \ca(\bd)$, $|f(0)|=1$,  satisfies the growth condition
$$
\log |f(z)|\le \frac K{(1-|z|)^p},
$$
where $p\ge 1$, then for any $\ep>0$
\begin{equation}\label{golub}
\sum_{\zeta\in Z(f)} (1-|\zeta|)^{p+1+\ep}\le C_0\cdot K,
\end{equation}
where the constant $C_0=C_0(p,\ep)$ depends on $p$ and $\ep$, see, e.g., Golubev \cite{go}.

Of course, the study of the zero distribution of analytic functions from other classes is much more involved; see, for instance, papers of
Korenblum \cite{ko1, ko2} on the zero distribution for functions from spaces $A^{-p}(\bd), A^{-\infty}(\bd)$. Interesting results
on zeros of functions from some Bergman-type spaces are given in Seip \cite{se}.

The above mentioned spaces of analytic functions are defined with the help of a radial (i.e., invariant with respect to rotations of
the unit disk) growth conditions. However, it turns out that one often needs to deal with classes of analytic functions subject to
{\it non-radial} growth relations. These classes appear, in particular, if one wants to study the distribution
of the discrete spectrum of non-self-adjoint perturbations for certain self-adjoint or unitary operators.

The study of such classes was initiated in \cite{bgk1}, and the main result therein looks as follows, see \cite[Theorem 0.2]{bgk1}. Given a finite
set $F=\{\xi_k\}_{k=1}^m$ on the unit circle $\bt=\{|z|=1\}$, let $\dist(z,F)=\min_k|z-\xi_k|$ denote the Euclidian distance between
a point $z\in\bd$ and $F$. In what follows $K$ is a positive constant.

\begin{theoremm}
Let $f\in \ca(\bd)$, $|f(0)|=1$, satisfy the growth condition
\begin{equation*}
\log|f(z)|\le\frac{K}{(1-|z|)^p\, \dist^q(z,F)}\,, \qquad z\in\bd, \quad p,q\ge0.
\end{equation*}
Then, for each $\ep>0$ there is a positive number $C_1=C_1(F,p,q,\ep)$  such that the following Blaschke-type condition holds:
\begin{equation}\label{btc0}
\sum_{\z\in Z(f)}
{(1-|\z|)^{p+1+\ep}}\, \dist^{(q-1+\ep)_+}(\z,F) \le C_1\cdot K, \qquad a_+:=\max(a,0).
\end{equation}
Moreover, in the case $p=0$ the term $(1-|\z|)^{p+1+\ep}$ can be replaced by $(1-|\z|)$.
\end{theoremm}

Theorem A effectively applies to the study of the discrete spectrum of complex perturbations of certain self-adjoint operators of mathematical physics in Demuth--Hansmann--Katriel \cite{DeHaKa, DeHaKa01}, Golinskii--Kupin \cite{goku1,goku2,goku3}, Dubuisson \cite{Du, Du2}, and Sambou \cite{Sa}.
We also mention recent interesting papers by Cuenin--Laptev--Tretter \cite{cuela}, Frank--Sabin \cite{frasa}, Frank \cite{fra15}, and Laptev--Safronov \cite{LaSa}
in this connection. For some extensions of this result to the case of arbitrary closed sets $F$ and subharmonic on $\bd$ functions $f$, and applications
in perturbation theory see Favorov--Golinskii \cite{fg1, fg2}.

\medskip

Let us go over to the  main results of the present paper which extend Theorem~A. Let $E=\{\z_j\}_{j=1}^n$ and
$F=\{\xi_k\}_{k=1}^m$ be two disjoint finite sets of distinct points on the unit circle $\bt$.

\begin{theorem}\label{th01}
Let $f\in \ca(\bd)$, $|f(0)|=1$, satisfy the growth condition
\begin{equation}\label{grow0}
\log|f(z)|\le\frac{K}{(1-|z|)^p}\,\frac{\dist^r(z,E)}{\dist^q(z,F)}\,, \qquad z\in\bd, \quad p,q,r\ge0.
\end{equation}
Then for every $\ep>0$, there is a positive number $C_2=C_2(E,F,p,q,r,\ep)$ such that
the following Blaschke-type condition holds:
\begin{equation}\label{btc1}
\sum_{\z\in Z(f)}
{(1-|\z|)^{p+1+\ep}}\,\frac{\dist^{(q-1+\ep)_+}(\z,F)}{\dist^{\min(p,r)}(\z,E)}\le C_2\cdot K.
\end{equation}
\end{theorem}
Of course, Theorem A is exactly Theorem \ref{th01} with $r=0$.

An obvious inequality for an arbitrary finite set $B=\{\b_j\}_{j=1}^n\subset\bt$
$$ c(B)\,\prod_{j=1}^n |z-\b_j|\le\dist(z,B)\le C(B)\,\prod_{j=1}^n |z-\b_j| $$
along with Theorem 0.3 from \cite{bgk1} prompt a more general statement.

\begin{theorem}\label{th02}
Let $f\in \ca(\bd)$, $|f(0)|=1$, satisfy the growth condition
\begin{equation}\label{grow}
\log|f(z)|\le\frac{K}{(1-|z|)^p}\,\frac{\prod_{j=1}^n
|z-\z_j|^{r_j}}{\prod_{k=1}^m |z-\xi_k|^{q_k}}\,, \qquad z\in\bd, \quad p,q_k,r_j\ge0.
\end{equation}
Then for every $\ep>0$, there is a positive number $C_3=C_3(E,F,p,\{q_k\},\{r_j\},\ep)$ such that
the following Blaschke-type condition holds:
\begin{equation}\label{btc2}
\sum_{\z\in Z(f)} (1-|\z|)^{p+1+\ep}\,\frac{\prod_{k=1}^m
|\z-\xi_k|^{(q_k-1+\ep)_+}}{\prod_{j=1}^n |\z-\z_j|^{\min(p,r_j)}}\le C_3\cdot K.
\end{equation}
\end{theorem}

Once again, in the case $p=0$ the factor $(1-|\z|)^{1+\ep}$ in \eqref{btc1} and \eqref{btc2} can be replaced by
$(1-|\z|)$.

\begin{remarkk}
An observation due to Hansmann--Katriel \cite{HaKa} applies in our setting. It turns out that the stronger assumption
\begin{equation*}
\log|f(z)|\le\frac{K|z|^\g}{(1-|z|)^p}\,\frac{\prod_{j=1}^n |z-\z_j|^{r_j}}{\prod_{k=1}^m |z-\xi_k|^{q_k}}\,,
\qquad z\in\bd, \quad p,q_k,r_j,\g\ge 0,
\end{equation*}
implies the stronger conclusion
\begin{equation*}
\sum_{\z\in Z(f)}
\frac{(1-|\z|)^{p+1+\ep}}{|\z|^{(\g-\ep)_+}}\,\frac{\prod_{k=1}^m
|\z-\xi_k|^{(q_k-1+\ep)_+}}{\prod_{j=1}^n |\z-\z_j|^{\min(p,r_j)}}\le C (E,F,p,\{q_k\},\{r_j\},\ep)\cdot K.
\end{equation*}
\end{remarkk}

The result of Theorem \ref{th01} can be extended in another direction involving arbitrary closed subsets $F$ of the unit circle. A key
ingredient in such extensions is the following quantitative characteristic of $F$ known as the {\it Ahern--Clark type} \cite{ahcla}:
\begin{equation*}
\a(F):=\sup\{\a\in\br: |\{t\in\bt:\, \dist(t,F)<x\}|=O(x^\a), \quad x\to +0\}.
\end{equation*}
Here $|A|$ denotes the Lebesgue measure of a measurable set $A\subset\bt$.

\begin{theorem}\label{th04}
Let $E=\{\z_j\}_{j=1}^n$ be a finite subset of $\bt$, $F\subset\bt$ be an arbitrary closed set, and $E\cap F=\emptyset$.
Let $f\in \ca(\bd)$, $|f(0)|=1$, satisfy the growth condition $\eqref{grow0}$.
Then for every $\ep>0$, there is a positive number $C_4=C_4(E,F,p,q,r,\ep)$ such that the following Blaschke-type condition holds:
\begin{equation}\label{btc12}
\sum_{\z\in Z(f)}
{(1-|\z|)^{p+1+\ep}}\,\frac{\dist^{(q-\a(F)+\ep)_+}(\z,F)}{\dist^{\min(p,r)}(\z,E)}\le C_4\cdot K.
\end{equation}
\end{theorem}

Clearly, Theorem \ref{th01} is a special case of the latter result, since $\a(F)=1$ for finite sets $F$.

\smallskip

As we will see later in Section \ref{s15}, inequalities \eqref{btc1}, \eqref{btc2} are in some sense ``local'' with respect to the singular points
$\{\zeta_j\}^n_{j=1}$ and $\{\xi_k\}^m_{k=1}$ on the unit circle, so we can restrict ourselves to the case $n=m=1$ and
$E=\{\z_0\}$, $F=\{\xi_0\}$.  The following ``one-point'' version of the main result will be crucial in the sequel.

\begin{theorem}\label{th03}
Let $\z_0,\xi_0\in\bt$, $\z_0\not=\xi_0$, and let $f\in \ca(\bd)$, $|f(0)|=1$, satisfy the growth condition
\begin{equation}\label{grow1}
\log|f(z)|\le \frac{K}{(1-|z|)^p}\,\frac{|z-\z_0|^r}{|z-\xi_0|^q}\,,  \qquad z\in\bd, \quad p,q,r\ge0.
\end{equation}
Then for every $\ep>0$, there is a positive number $C_5=C_5(\z_0,\xi_0,p,q,r,\ep)$ such that the following inequality holds:
\begin{equation}\label{btc11}
\sum_{\z\in Z(f)}
(1-|\z|)^{p+1+\ep}\,\frac{|\z-\xi_0|^{(q-1+\ep)_+}}{|\z-\z_0|^{\min(p,r)}}\le C_5\cdot K.
\end{equation}
\end{theorem}

\smallskip

The paper is organized in a straightforward manner. The preliminaries are given in Section \ref{s11}. In Sections \ref{s12} and \ref{s15}
we prove Theorem \ref{th03} and then deduce the general statements in Theorems \ref{th02} and \ref{th04} from
this one-point version. Some further results (the analogs for the upper half-plane and the plane with a cut) are given in Section \ref{s16}.

To keep the notation reasonably simple and consistent, we usually number the constants $C_k$ appearing in the formulations of theorems,
propositions, etc. The constants $C$ arising in the proofs are generic, i.e., the same symbol does not necessarily denote the same
constant in different occurrences.

\section{Conformal mappings, Pommerenke lemma and Stolz angles}\label{s11}

We start with some general preliminaries from Complex Analysis.

The known distortion inequalities \cite[Corollary 1.4]{po} play a key role in what follows.

\begin{lemma}\label{pom}
Let $\oo$ be a bounded, simply connected domain with the boundary $\pt\oo$, and $\p$ be a conformal mapping of $\oo$ onto $\bd$. Then
\begin{equation}\label{boundpom}
\frac12\dist(w,\pt\oo)\cdot|\p'(w)|\le 1-|\p(w)|\le 4\dist(w,\pt\oo)\cdot|\p'(w)|, \qquad w\in\oo.
\end{equation}
\end{lemma}

This result will be applied in the following situation, wherein the bounds on derivatives
can be specified. It is related to the Stolz angle with the vertex at $\z_0\in\bt$, that is, a domain
inside the unit disk of the form
\begin{equation}\label{stang}
\css_A(\z_0):=\left\{z\in\bd: \frac{|z-\z_0|}{1-|z|}<A\right\}, \qquad A>1.
\end{equation}
When $\z_0=1$, we use the abbreviation $\css_A:=\css_A(1)$, see Figure \ref{fig01}.
The interior angle of $\css_A$ at $1$ equals $2\o$, $0<\o:=\arccos A^{-1}<\pi/2$.  The Stolz angles $\{\css_A\}_{A>1}$
form an increasing family of sets which exhaust the unit disk as $A\to\infty$. The boundary  of $\css_A$ is denoted by
$\pt\css_A$.

\begin{figure}[htbp]
\includegraphics[width=10cm]{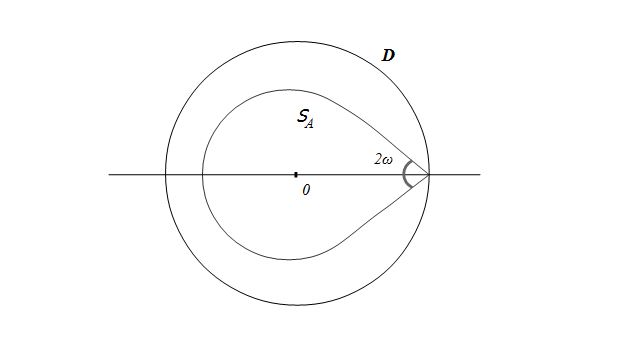}
\caption{Stolz angle $\css_A=\css_A(1), \  A>1, \ \o=\arccos A^{-1}$.}
\label{fig01}
\end{figure}

Let $\p_A$ denote the conformal mapping $\p_A: \css_A\to\bd$, $\p_A(0)=0$, $\p_A(1)=1$.
The following result provides a local uniform bound for its derivative $\p_A'$.

\begin{lemma}\label{dercm}
Let
\begin{equation*}
\a=\a_A:=\frac{\pi}{2\o}=\frac{\pi}{2\arccos A^{-1}}\,, \quad \a>1.
\end{equation*}
Then the following bounds hold uniformly for $A\ge2$:
\begin{equation}\label{boundder}
\frac1{16}<\frac{|\p'_A(z)|}{|z-1|^{\a-1}}<48, \qquad z\in\css_A^+:=\css_A\bigcap\left\{|z-1|<\frac1{16}\right\}.
\end{equation}
\end{lemma}

\begin{proof}
We just sketch the proof which is rather standard.
Let $\psi:\bd\to\bc_r:=\{z: \re z>0\}$ be the linear-fractional mapping of $\bd$ onto the right half-plane,
$\psi(0)=1$, $\psi(1)=\infty$.
A crucial observation is that $\psi$ maps $\css_A$ onto the interior $H_i$ of the right branch of the hyperbola
\footnote{We thank D. Tulyakov for this remark.}
$$
H: \ \frac{x^2}{\cos^2\o}-\frac{y^2}{\sin^2\o}=1, \quad z=x+iy.
$$
Set $\mathbb C_\pm=\{z:\pm \im z>0\}$, $H_\pm=H_i\cap \mathbb C_\pm$, and define
$\phi_1(z)=z+\sqrt{z^2-1}=\exp(\text{\rm arch}\,z)$,
$\phi_1:H_\pm\to A_\pm=\{re^{i\theta}:r>1,\,\pm\theta\in(0,\o)\}$, $\phi_2(z)=(z^{\pi/\o}+z^{-\pi/\o})/2$,
$\phi_2:A_\pm\to \mathbb C_\pm$, $\phi_3(z)=\sqrt{1+z}$,  $\phi_3:\mathbb C_\pm\to \mathbb C_\pm\cap \bc_r$.
Since $\phi_3\circ \phi_2\circ \phi_1$ extends continuously to $H_i\cap\mathbb R$, we obtain a conformal map
$\phi:H_i\to \bc_r$, $\phi(1)=\sqrt2$, $\phi(\infty)=\infty$. Finally, if $\phi_4(z)=(z-\sqrt2)/(z+\sqrt2)$, then
$$
\p_A=\phi_4\circ\phi\circ\psi=
\biggl(\frac{(1\pm\sqrt{z})^\alpha-(1\mp\sqrt{z})^\alpha}{(1\pm\sqrt{z})^\alpha+(1\mp\sqrt{z})^\alpha}\biggr)^2=
1-\frac{4(1-z)^\alpha}{((1\pm\sqrt{z})^\alpha+(1\mp\sqrt{z})^\alpha)^2}
$$
(see also Lavrent'ev--Shabat \cite[Chapter 2.3.36]{lasha}).
Next,
\begin{equation}\label{deriv}
|\p_A'(z)|=\frac{4\a|1-z|^{\a-1}}{\sqrt{|z|}}\cdot
\frac{|(1\pm\sqrt{z})^\alpha-(1\mp\sqrt{z})^\alpha|}{|(1\pm\sqrt{z})^\alpha+(1\mp\sqrt{z})^\alpha|^3}.
\end{equation}
Since $\a\le3/2$ for $A\ge2$, the elementary bounds
\begin{gather*}
\frac12\le \sqrt{|z|}\le1,\\
1\le |(1\pm\sqrt{z})^\alpha+(1\mp\sqrt{z})^\alpha|\le 4,\\
1\le |(1\pm\sqrt{z})^\alpha-(1\mp\sqrt{z})^\alpha|\le 4,
\end{gather*}
valid for $z\in\css_A^+$ yield \eqref{boundder}.
\end{proof}

The following simple relation between two Stolz angles is casted as a lemma for convenience only; its elementary proof is omitted.
\begin{lemma}\label{twosa}
Let $A<B$, so $\css_A\subset\css_B$. Then, for $z\in\css_A$,
\begin{equation}\label{saAB}
\frac{B-A}{B+1}\,(1-|z|)\le\dist(z,\pt\css_B)<1-|z|.
\end{equation}
\end{lemma}

\medskip

For $0<a<1$, consider a nested family of domains (curvilinear quadrangles) $\{L_a\}$, see Figure \ref{fig02},
\begin{equation}\label{lune}
L_{a_1}\subset L_{a_2}\subset\bd, \qquad 0<a_1<a_2<1.
\end{equation}
We denote by $\eta=\eta_a$ the conformal mapping of $L_a$ onto $\bd$ with normalization $\eta(0)=0$, $\eta(1)=1$,
and write $\eta_j$, $j=1,2$ for the domains $L_{a_j}$. Although there is no explicit formula for $\eta$, it is easily seen from
\cite[Theorem 3.9]{po} that both $\eta$ and $\eta'$ have continuous extensions on the closure $\overline{L}_a$, and so
\begin{equation}\label{conf2}
|\eta'(z)|\le c(a), \qquad z\in L_a.
\end{equation}

\begin{figure}[htbp]
\includegraphics[width=10cm]{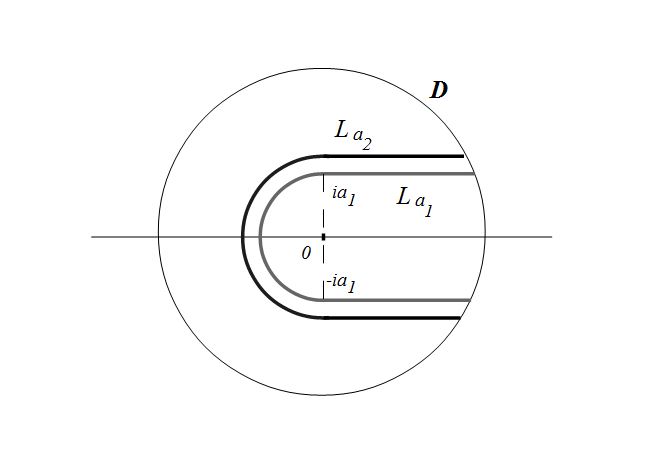}
\caption{Domains $L_{a_1},\, L_{a_2}, 0<a_1<a_2$.}
\label{fig02}
\end{figure}

The relations below follow directly from \eqref{conf2} and Lemma \ref{pom}. First,
\begin{equation}\label{bounlune1}
1-|\eta_2(z)| \le c_1(a_2)(1-|z|), \qquad z\in L_{a_2},
\end{equation}
holds with some positive constant $c_1(a_2)$. Next, since $1$ is a regular point for $\eta$ ($\eta$ is analytic at some 
neighborhood of $1$),
\begin{equation}\label{bounlune2}
c_2(a_2)\le \frac{|1-\eta_2(z)|}{|1-z|}\le c_3(a_2),  \qquad z\in L_{a_2},
\end{equation}
holds with some positive constants $c_j(a_2)$, $j=2,3$. Then, there are positive constants $c_j=c_j(a_1,a_2), \, j=4,5,$ such that
\begin{equation}\label{bounlune3}
c_4(a_1,a_2)\le \frac{1-|\eta_2(z)|}{1-|z|}\le c_5(a_1,a_2), \qquad z\in L_{a_1}.
\end{equation}
We will exploit these relations later in Section \ref{s15}.

\section{Proof of Theorem \ref{th03} for $q=0$}\label{s12}

Without loss of generality we assume that $\z_0=1$.
By \eqref{grow1},
$$ \log|f(z)|\le \frac{2^rK}{(1-|z|)^p}, \qquad z\in\bd,$$
and, by \eqref{golub},  for each $\ep>0$,
\begin{equation}\label{bgk1.02}
\sum_{\z\in Z(f)} (1-|\z|)^{p+1+\ep}\le C(p,r,\ep)\cdot K.
\end{equation}

To clarify the local character of the problem, put
\begin{equation}\label{zeropm}
Z^+(f):=Z(f)\bigcap\left\{|z-1|<\frac1{16}\right\}, \ \
Z^-(f):=Z(f)\bigcap\left\{|z-1|\ge\frac1{16}\right\},
\end{equation}
so that with $s:=\min(p,r)$ we have
\begin{equation*}
\sum_{\z\in Z(f)} \frac{(1-|\z|)^{p+1+\ep}}{|1-\z|^{s}}=
\left\{\sum_{\z\in Z^+(f)} + \sum_{\z\in Z^-(f)}\right\}
\frac{(1-|\z|)^{p+1+\ep}}{|1-\z|^{s}}=\ss^+ + \ss^-.
\end{equation*}
The bound for $\ss^-$ follows directly from \eqref{bgk1.02} and the inequality $|1-\z|\ge 1/16$:
\begin{equation}\label{sigma2}
\ss^-=\sum_{\z\in Z^-(f)}\frac{(1-|\z|)^{p+1+\ep}}{|1-\z|^{s}} \le C(p,r,\ep)\cdot K, \quad p,r\ge0.
\end{equation}
Thus, the main problem is to prove \eqref{btc11} for $\ss^+$.

\subsection{Case $p\le r$}\label{s13}

Given a function $f\in \ca(\bd)$ and a number $A\ge2$, put (see \eqref{zeropm} and \eqref{boundder})
\begin{equation}\label{zprime}
Z_A(f):=Z^+(f)\cap\css_A=Z(f)\cap\css_A^+.
\end{equation}

\nt
{\bf Step 1}.
Let $\psi_A=\p_A^{(-1)}$ be the conformal mapping from $\bd$ onto $\css_A$, $\psi_A(0)=0$.
Set $f_A=f(\psi_A)$. Then $f_A\in\ca(\bd)$, $|f_A(0)|=1$ and by \eqref{grow1} we have
$$
\log|f_A(w)|\le K\,\frac{|1-\psi_A(w)|^r}{(1-|\psi_A(w)|)^p}\le 2^{r-p}A^p\cdot K,
\qquad w\in\bd.
$$
The Poisson--Jensen formula implies
\begin{equation}\label{poijen}
\sum_{w\in Z(f_A)} (1-|w|)\le 2^{r-p}A^p\cdot K.
\end{equation}
However, $Z(f_A)=\p_A(Z(f)\cap\css_A)$, and so
\begin{equation}
\sum_{w\in Z(f_A)} (1-|w|)=\sum_{\z\in Z(f)\cap\css_A} (1-|\p_A(\z)|)\le 2^{r-p}A^p\cdot K.
\end{equation}
By Lemma \ref{pom},
$$
\sum_{\z\in Z(f)\cap\css_A} |\p'_A(\z)|\cdot\dist(\z,\pt\css_A)\le 2^{r-p+1}A^p\cdot K,
$$
and, since $Z(f)\cap\css_A\supset Z_A(f)=Z(f)\cap\css_A^+$, it follows from Lemma \ref{dercm} that for $A\ge 2$
\begin{equation*}
\begin{split}
\sum_{\z\in Z(f)\cap\css_A} |\p'_A(\z)|\cdot\dist(\z,\pt\css_A) &\ge \sum_{\z\in Z_A(f)} |\p'_A(\z)|\cdot\dist(\z,\pt\css_A) \\
&\ge \frac1{32}\,\sum_{\z\in Z_A(f)} |\z-1|^{\a-1}\cdot\dist(\z,\pt\css_A).
\end{split}
\end{equation*}
Hence,
\begin{equation}\label{step1}
\sum_{\z\in Z_A(f)} |1-\z|^{\a-1}\cdot\dist(\z,\pt\css_A)\le 2^{r-p+6}A^p\cdot K, \quad \a=\frac{\pi}{2\arccos A^{-1}}\,.
\end{equation}

\nt
{\bf Step 2}. In what follows $A=A_k=2^k$, $k\in\bn$, so the Stolz angles $\css_k:=\css_{A_k}$
(with a little abuse of notation) exhaust the unit disk, as $k\to\infty$.
Relation \eqref{step1} with $A=A_{k+1}$ takes the form
\begin{equation*}
\sum_{\z\in Z_{k+1}} |1-\z|^{\a_{k+1}-1}\cdot\dist(\z,\pt\css_{k+1})\le 2^{kp+r+6}K,
\quad Z_k:=Z_{A_k}(f)=Z^+(f)\cap\css_{k},
\end{equation*}
see \eqref{zeropm}, \eqref{zprime}, or, since $Z_k\subset Z_{k+1}$,
\begin{equation}\label{step11}
\sum_{\z\in Z_k} |1-\z|^{\b_{k+1}}\cdot\dist(\z,\pt\css_{k+1}) \le 2^{kp+r+6}K, \quad \b_{k+1}:=\a_{k+1}-1.
\end{equation}

To apply Lemma \ref{twosa} with $A=2^k$, $B=2^{k+1}$, notice that
$$ \frac{B-A}{B+1}=\frac{2^{k+1}-2^k}{2^{k+1}+1}\ge\frac25\,, $$
so \eqref{step11} entails
\begin{equation}\label{bound1}
\sum_{\z\in Z_k} (1-|\z|)|1-\z|^{\b_{k+1}}\le 5K\cdot2^{kp+r+5}=C(r)\,  2^{kp}\cdot K,
\end{equation}
for $k\in\bn$. It is convenient to deal with a chain of inequalities
\begin{equation*}
\sum_{\z\in Z_k\backslash Z_{k-1}} (1-|\z|)|1-\z|^{\b_{k+1}}\le C(r)\, 2^{kp}\cdot K, \quad k\in\bn,
\qquad Z_0:=\emptyset.
\end{equation*}
Take an arbitrary $0<\ep<1/16$ and write
\begin{equation}\label{bound2}
\frac1{2^{k(p+\ep)}}\,\sum_{\z\in Z_k\backslash Z_{k-1}} (1-|\z|)|1-\z|^{\b_{k+1}}
\le  C(r)\, 2^{-\ep k}\cdot K.
\end{equation}
On the set $Z_k\backslash Z_{k-1}$ we have
$$ 2^{-k}< \frac{1-|\z|}{|1-\z|}\le 2^{-k+1}, \qquad
\left(\frac{1-|\z|}{|1-\z|}\right)^{p+\ep}\le\frac{2^{p+\ep}}{2^{k(p+\ep)}}\,, $$
and so
\begin{equation}\label{bound3}
\sum_{\z\in Z_k\backslash Z_{k-1}} \frac{(1-|\z|)^{p+\ep+1}}{|1-\z|^{p+\ep}}\,
|1-\z|^{\b_{k+1}}\le C(p,r)\, 2^{-\ep k}\cdot K.
\end{equation}

\nt
{\bf Step 3}.  We have
$$ \a_k=\frac{\pi}{2\arccos 2^{-k}}, \qquad \b_k=\a_k-1=\frac{\arcsin 2^{-k}}{\arccos 2^{-k}}, $$
and as $x\le\arcsin x\le \pi x/2$ for $0\le x\le 1$, and $\arccos 1/2=\pi/3$, we see that
\begin{equation}\label{alpha}
\frac2{\pi}\le 2^k\b_k\le \frac32\,.
\end{equation}
By definition, $\b_k\searrow 0$ as $k\to\infty$. Now, choose $k_0=k_0(\ep)$ from the relations
\begin{equation}\label{kzero}
2^{-k_0-1}\le\ep< 2^{-k_0},
\end{equation}
and hence
\begin{equation}\label{inclusions}
\css_{k_0}\subset\css_{1/\ep}\subset\css_{k_0+1}, \quad
Z_{k_0}\subset Z^+(f)\cap\css_{1/\ep}\subset Z_{k_0+1}\,.
\end{equation}
By \eqref{alpha} and \eqref{kzero}, one has for $k\ge k_0+1$
$$
\b_{k+1}\le\b_{k_0+2}\le\frac32\,2^{-k_0-2}<\frac34\,\ep.
$$
Let $z\in Z^+(f)$. Since $|1-z|<1/16$, we see that $|1-z|^{\b_{k+1}}\ge|1-z|^\ep$. Hence, \eqref{bound3} implies that
\begin{equation}\label{bound4}
\sum_{\z\in Z_k\backslash Z_{k-1}} \frac{(1-|\z|)^{p+\ep+1}}{|1-\z|^{p}}
\le C(p,r)\, 2^{-\ep k}\cdot K, \quad k\ge k_0+1.
\end{equation}
Summation over $k$ from $k=k_0+1$ to infinity gives
\begin{equation}\label{bound5}
\sum_{\z\in Z^+(f)\backslash Z_{k_0}} \frac{(1-|\z|)^{p+\ep+1}}{|1-\z|^{p}}\le C(p,r,\ep)\cdot K.
\end{equation}
Next, write
$$ Z^+(f)=(Z^+(f)\cap\css_{1/\ep})\bigcup (Z^+(f)\cap\css_{1/\ep}^c), \quad \css_{1/\ep}^c:=\bd\backslash\css_{1/\ep}. $$
By \eqref{inclusions}, $Z^+(f)\backslash Z_{k_0}\supset Z^+(f)\cap\css_{1/\ep}^c$,
so \eqref{bound5} provides
\begin{equation}\label{bound6}
\sum_{\z\in Z^+(f)\cap\css_{1/\ep}^c} \frac{(1-|\z|)^{p+\ep+1}}{|1-\z|^{p}}\le C(p,r,\ep)\cdot K.
\end{equation}

On the other hand, put $k=k_0+1$  in \eqref{bound1}. By  \eqref{kzero},
$\b_{k_0+2}<\ep$, and \eqref{inclusions} implies that
\begin{equation}\label{bound7}
\sum_{\z\in Z^+(f)\cap\css_{1/\ep}} (1-|\z|) |1-\z|^{\ep}\le
\sum_{\z\in Z_{k_0+1}} (1-|\z|) |1-\z|^{\ep}\le C(r,\ep)\cdot K.
\end{equation}
The sum of \eqref{bound6} and \eqref{bound7} gives
\begin{equation}\label{bound7.1}
\sum_{\z\in Z^+(f)\cap\css_{1/\ep}^c} \frac{(1-|\z|)^{p+\ep+1}}{|1-\z|^{p}} +
\sum_{\z\in Z^+(f)\cap\css_{1/\ep}} (1-|\z|) |1-\z|^{\ep}\le C(p,r,\ep)\cdot K,
\end{equation}
and it remains to note again that the inequality  $|1-z|\ge1-|z|$ for $z\in\bd$ implies
$$
(1-|z|)|1-z|^\ep \ge\frac{(1-|z|)^{p+\ep+1}}{|1-z|^{p}}\,.
$$
Finally,
\begin{equation}\label{bound8}
\ss^+=\sum_{\z\in Z^+(f)} \frac{(1-|\z|)^{p+\ep+1}}{|1-\z|^{p}}\le C(p,r,\ep)\cdot K.
\end{equation}
Note that now $p=s=\min(p,r)$. A combination of \eqref{bound8} and \eqref{sigma2} completes the proof
of Theorem \ref{th03} in the case $q=0$, $p\le r$.

\subsection{Case $p>r$}\label{s14}

Let $f\in \ca(\bd)$, $|f(0)|=1$, satisfy
\begin{equation}
\log|f(z)|\le K\,\frac{|1-z|^r}{(1-|z|)^p}\,, \qquad z\in\bd,
\end{equation}
with  $0\le r<p$.
Recalling the notation $f_A=f(\psi_A)$ (see Section \ref{s13}) we have
\begin{equation*}
\log|f_A(w)| 
\le K\,\frac{|1-\psi_A(w)|^r}{\bigl(1-|\psi_A(w)|\bigr)^r}
\cdot\frac1{\bigl(1-|\psi_A(w)|\bigr)^{p-r}} 
\le \frac{K\,A^r}{\bigl(1-|\psi_A(w)|\bigr)^{p-r}}\,.
\end{equation*}
By the Schwarz lemma, $|\psi_A(w)|\le|w|$, and so
\begin{equation}
\log|f_A(w)|\le \frac{K\,A^r}{(1-|w|)^{p-r}}\,.
\end{equation}
As above in \eqref{bgk1.02}, we get for each $\ep>0$
\begin{equation}\label{sigma21}
\sum_{w\in Z(f_A)} (1-|w|)^{\g}\le C(p,r,\ep)\, A^r\cdot K,
\end{equation}
where $\g=\g(p,r,\ep):=p-r+1+\ep$.
So we come to \eqref{poijen} with exponent $\g$ instead of $1$.

The rest is essentially the same as in the argument for the case $p\le r$. For instance,
\eqref{step1} becomes
\begin{equation}\label{step111}
\sum_{\z\in Z_A(f)} |1-\z|^{\g(\a-1)}\cdot\dist^\g(\z,\pt\css_A)\le C(p,r)\, A^r\cdot K,
\end{equation}
and \eqref{bound3} turns into
\begin{equation}\label{bound31}
\sum_{\z\in Z_k\backslash Z_{k-1}} \frac{(1-|\z|)^{p+1+2\ep}}{|1-\z|^{r+\ep}}\,
|1-\z|^{\g\b_{k+1}}\le C(p,r)\; 2^{-\ep k}\cdot K.
\end{equation}
The choice of $k_0$ is somewhat different from \eqref{kzero}:
\begin{equation*}
2^{-k_0-1}\le \frac{\ep}{p-r+2} <2^{-k_0},
\end{equation*}
and again $\g\b_{k+1}\le\ep$ for $k\ge k_0+1$. Thereby we come to
\begin{equation}\label{bound51}
\sum_{\z\in Z^+(f)\backslash Z_{k_0}} \frac{(1-|\z|)^{p+1+2\ep}}{|1-\z|^{r}} \le C(p,r,\ep)\cdot K;
\end{equation}

compare this inequality to \eqref{bound5}. Finally,
\begin{equation}\label{bound81}
\sum_{\z\in Z^+(f)} \frac{(1-|\z|)^{p+1+2\ep}}{|1-\z|^{r}}\le C(p,r,\ep)\cdot K.
\end{equation}
A combination of \eqref{bound81} and \eqref{sigma2} completes the proof of
Theorem \ref{th03} for $q=0$.

\section{Proofs of Theorem \ref{th03} with $q>0$,  Theorems \ref{th02} and \ref{th04}}
\label{s15}

We proceed with a local version of the result obtained in Section \ref{s12}, see also Favorov--Golinskii \cite{fg2}.

\begin{proposition}\label{local}
Given the quadrangle $L_{a_2}$ on Figure $\ref{fig02}$, let $g\in \ca(L_{a_2})$ satisfy
\begin{equation}\label{loc1}
\log|g(w)|\le K\,\frac{|1-w|^r}{(1-|w|)^p}\,, \qquad w\in L_{a_2}, \quad p,r\ge0.
\end{equation}
Then for every $\ep>0$ and every $0<a_1<a_2$ there exists a positive constant
$C=C(p,r,\ep; a_1,a_2)$ such that
\begin{equation}\label{loc2}
\sum_{\z\in Z(g)\cap L_{a_1}} \frac{(1-|\z|)^{p+1+\ep}}{|\z-1|^{s}}\le C \cdot K, \qquad s=\min(p,r).
\end{equation}
\end{proposition}
\begin{proof}
Recall that $\eta_2$ stands for the normalized conformal map from $L_{a_2}$ onto $\bd$. Put
$f:=g\circ \eta^{-1}_2$, so 
$$ \log|f(z)|\le K\,\frac{|1-\eta_2^{-1}(z)|^r}{(1-|\eta_2^{-1}(z)|)^p}\,, \qquad z\in\bd. $$
In view of \eqref{bounlune1}, \eqref{bounlune2} we have
$$
\log|f(z)|\le C K\,\frac{|1-z|^r}{(1-|z|)^p}\,, \qquad z\in\bd, \quad p,r\ge0.
$$
By the result obtained in Section \ref{s12}, for every $\ep>0$
$$
\sum_{v\in Z(f)} \frac{(1-|v|)^{p+1+\ep}}{|1-v|^{s}}=
\sum_{\z\in Z(g)} \frac{(1-|\eta_2(\z)|)^{p+1+\ep}}{|1-\eta_2(\z)|^{s}}\le C(p,r,\ep; a_2)\cdot K, \,\,\, s=\min(p,r),
$$
and moreover, for $0<a_1<a_2$
$$ \sum_{\z\in Z(g)\cap L_{a_1}} \frac{(1-|\eta_2(\z)|)^{p+1+\ep}}{|1-\eta_2(\z)|^{s}}\le C(p,r,\ep; a_1,a_2)\cdot K. $$
The result now follows from \eqref{bounlune2} and \eqref{bounlune3}.
\end{proof}

\smallskip
\nt
{\it Proof of Theorem \ref{th03}.}
Recall that, by convention, $\z_0=1$. To complete the proof of Theorem \ref{th03}, we note that
\eqref{grow1} implies \eqref{loc1} locally inside the domain $L_{a}$ with $4a=|1-\xi_0|$ and with $K$ replaced by 
$C(\zeta_0,\xi_0)\cdot K$.  Put  $\rho:=(q-1+\ep)_+$.
By Proposition~\ref{local},
\begin{equation}\label{loc21}
\begin{split}
\sum_{\z\in Z(f)\cap L_{a/2}}(1-|\z|)^{p+1+\ep}\,\frac{|\z-\xi_0|^{\rho}}{|\z-1|^{s}} &\le
2^{\rho}\sum_{\z\in Z(f)\cap L_{a/2}} \frac{(1-|\z|)^{p+1+\ep}}{|\z-1|^{s}} \\ &\le
C (p,q,r,\xi_0,\ep)\cdot K.
\end{split}
\end{equation}

On the other hand, condition \eqref{grow1} implies the global bound
\begin{equation*}
\log|f(z)|\le \frac{2^r K}{(1-|z|)^p\,|z-\xi_0|^q}\,,  \qquad z\in\bd,
\end{equation*}
and so
\begin{equation}\label{loc22}
\begin{split}
&{} \sum_{\z\in Z(f)\backslash L_{a/2}}(1-|\z|)^{p+1+\ep}
\frac{|\z-\xi_0|^{\rho}}{|\z-1|^{s}}\le C \sum_{\z\in Z(f)\backslash L_{a/2}}
(1-|\z|)^{p+1+\ep}|\z-\xi_0|^{\rho} \\ &\le C \sum_{\z\in Z(f)}
(1-|\z|)^{p+1+\ep}|\z-\xi_0|^{\rho} \le C(p,q,r,\xi_0,\ep) \cdot K.
\end{split}
\end{equation}
The latter inequality follows from Theorem A. The combination of \eqref{loc21} and \eqref{loc22} completes
the proof of Theorem \ref{th03}. \hfill $\Box$

\smallskip
\nt
{\it Proof of Theorem \ref{th02}}.
We follow the line of reasoning of the above proof. In view of \eqref{grow} one has the bound,
which holds inside the turned quadrangle
$$
L_a(\z_i)=\z_i\,L_a, \qquad a:=\frac12\,\min_{1\le j\le n} \dist(\z_j, E\backslash\{\z_j\}). $$
Precisely,
\begin{equation}\label{grow11}
\log|f(z)|\le C\, K\frac{|z-\z_i|^{r_i}}{(1-|z|)^p}\,, \qquad z\in L_a(\z_i), \quad i=1,2,\ldots,n.
\end{equation}
By Proposition \ref{local}, for $i=1,2,\ldots,n$ and $s_i=\min(p,r_i)$
\begin{equation}\label{btc3}
\sum_{\z\in Z(f)\cap L_{a/2}(\z_i)}
\frac{(1-|\z|)^{p+1+\ep}}{\prod_{j=1}^n |\z-\z_j|^{s_j}}\le C\cdot\sum_{\z\in Z(f)\cap L_{a/2}(\z_i)}
\frac{(1-|\z|)^{p+1+\ep}}{|\z-\z_i|^{s_i}}\le C\cdot K.
\end{equation}

On the other hand, if we ``ignore'' the product in the numerator of \eqref{grow}, we get the global bound
\begin{equation*}
\log|f(z)|\le\frac{K}{(1-|z|)^p\,\prod_{k=1}^m |z-\xi_k|^{q_k}}\,, \qquad z\in\bd,
\end{equation*}
and \cite[Theorem 0.2]{bgk1} gives
\begin{equation}\label{btc4}
\sum_{\z\in Z(f)}(1-|\z|)^{p+1+\ep}\,\prod_{k=1}^m |\z-\xi_k|^{(q_k-1+\ep)_+}\le C\cdot K.
\end{equation}
As above, the combination of \eqref{btc3} and \eqref{btc4} yields \eqref{btc2}, as claimed.\hfill $\Box$

\smallskip
\nt
{\it Proof of Theorem \ref{th04}}.
The argument is close to the one above. Within the domain $L_a(\z_i)$ with
$$
a:=\frac12\,\min_{1\le j\le n} \dist(\z_j, F\cup E\backslash\{\z_j\}), $$
the effect of the second factor in the denominator of  $\eqref{grow0}$ is negligible. Therefore, as above in \eqref{btc3},
we have  with $s=\min(p,r)$
\begin{equation}\label{btc13}
\sum_{\z\in Z(f)\cap L_{a/2}(\z_i)}
\frac{(1-|\z|)^{p+1+\ep}}{\dist^s(\z,E)}\le \sum_{\z\in Z(f)\cap L_{a/2}(\z_i)}
\frac{(1-|\z|)^{p+1+\ep}}{|\z-\z_i|^{s}}\le C\cdot K.
\end{equation}

The global bound now looks as
\begin{equation}\label{btc14}
\log|f(z)|\le \frac{K}{(1-|z|)^p\,\dist^q(z,F)}\,, \qquad z\in\bd,
\end{equation}
The Blaschke-type condition for $f$ in \eqref{btc14} with $p=0$ is a particular case of \cite[Theorem 3]{fg2}:
\begin{equation}\label{btc15}
\sum_{\z\in Z(f)} (1-|\z|)\, \dist^\rho(\z,F)\le C\cdot K, \qquad  \rho:=(q-\a(F)+\ep)_+.
\end{equation}
There is a standard way to carry the later result over to the case $p>0$, see the proof of
Theorem 0.2 in \cite{bgk1}. For the sake of completeness we outline the idea of this method.

Consider the sequence of functions
$$
f_n(z):=f(\l_n z), \qquad \l_n:=1-2^{-n}, \quad n\in\bn.
$$
By \eqref{btc14} and elementary inequality $\dist(z,F)\le 2\dist(\l_n z,F)$ we have
$$
\log |f_n(z)|\le \frac{2^{np+q}K}{\dist^q(z,F)}\,, \qquad z\in\bd.
$$
The latter is \eqref{btc14} with $p=0$, so, in view of \eqref{btc15},
\begin{equation}\label{btc16}
\sum_{j: |\z_j(f)|\le \l_{n-1}} (1-|\z_j(f_n)|)\, \dist^\rho(\z_j(f_n),F)\le C2^{np}\cdot K,
\end{equation}
where $\z_j(f)$, $\z_j(f_n)$ are the zeros of $f$ and $f_n$, respectively, so $\z_j(f)=\z_j(f_n)/\l_n$.

To obtain the lower bound of the LHS in \eqref{btc16}, we note that $|\z_j(f)|\le \l_{n-1}$ implies that
$$ 1-|\z_j(f_n)|=1-\frac{|\z_j(f)|}{\l_n}\ge\frac{1-|\z_j(f)|}2\,, \quad \dist(\z_j(f_n),F)\ge\frac12 \dist(\z_j(f),F), $$
and hence
$$ \sum_{\l_n<|\z_j(f)|\le \l_{n+1}}  (1-|\z_j(f)|)\, \dist^\rho(\z_j(f),F)\le C2^{np}\cdot K. $$
Since now
$$ 1-|\z_j(f)|<1-\l_n=2^{-n}, \qquad (1-|\z_j(f)|)^{p+\ep+1}<2^{-n(p+\ep)}, $$
the summation over $n$ leads to
\begin{equation}\label{btc17}
\sum_{\z\in Z(f)} {(1-|\z|)^{p+1+\ep}}\,\dist^{(q-\a(F)+\ep)_+}(\z,F)\le C\cdot K,
\end{equation}
which is the Blaschke-type condition for the functions $f$ in \eqref{btc14} with $p>0$.
Again, a combination of \eqref{btc13} and \eqref{btc17} gives \eqref{btc12}, as claimed. \hfill $\Box$

\section{Some further Blaschke-type conditions}\label{s16}

\subsection{Generalized Stolz domains}

There is a seemingly more general form of the Blaschke-type condition \eqref{btc1} which states that, under
assumption \eqref{grow0}, for every $0\le \tau'<\tau$ there is a positive constant $C=C(E,F,p,q,r,\tau,\tau')$ such that
\begin{equation}\label{btc18}
\sum_{\z\in Z(f)}
{(1-|\z|)^{p+1+\tau}}\,\frac{\dist^{(q-1+\tau)_+}(\z,F)}{\dist^{\min(p,r)+\tau'}(\z,E)}\le C\cdot K.
\end{equation}
In fact, it is a direct consequence of \eqref{btc1} with $\ep=\tau-\tau'$, since
$$ \frac{1-|\z|}{\dist(\z,E)}\le 1, \qquad \z\in\bd. $$
However, it turns out that in some instances \eqref{btc18} holds with $\tau'=\tau$, see Corollary \ref{genbtc}.

\medskip
Recall the notation $\css_A(\zeta_0), \ \zeta_0\in \bt, A>0$ introduced in \eqref{stang}.
In the proof of Theorem \ref{th03}, we actually obtained a little stronger conclusion than the claimed one.

\begin{proposition}\label{r03}
Let $f\in \ca(\bd)$ be a function satisfying the assumptions of Theorem \ref{th03}. Then for each $0\le\t'<\t$ and $\ep=\t-\t'>0$
there is a positive number $C_6=C_6(\zeta_0,\xi_0, p,q,r,\t,\t')$ such that the following condition holds:
\begin{eqnarray}\label{btc011}
&&\sum_{\z\in Z(f)\cap\css_{1/\ep}^c(\z_0)} \frac{(1-|\z|)^{p+\tau+1}  |\zeta-\xi_0|^{(q-1+\tau)_+}}{|\z-\z_0|^{\min(p,r)+\t'}}\\
&& \qquad \qquad +\sum_{\z\in Z(f)\cap\css_{1/\ep}(\z_0)} (1-|\z|)^{p+1+\ep} |\zeta-\xi_0|^{(q-1+\ep)_+} \le C_6\cdot K. \nonumber
\end{eqnarray}
\end{proposition}

Obviously, inequality \eqref{btc011} reads as
\begin{equation*}
\sum_{\z\in Z(f)\cap\css_{1/\ep}^c(\z_0)} \frac{(1-|\z|)^{p+\tau+1}}{|\z-\z_0|^{\min(p,r)+\t'}}+
\sum_{\z\in Z(f)\cap\css_{1/\ep}(\z_0)} (1-|\z|)^{p+1+\ep}\le C_6\cdot K.
\end{equation*}
when $q=0$. Of course, the above remark also holds for Theorems \ref{th01}, \ref{th02}.

To get sharper results we could replace summation along the Stolz angles by that along larger approach domains.
For simplicity, we formulate here just the result for one point $\z_0=1$.

\begin{theorem}\label{th031}
Let $f\in \ca(\bd)$, $|f(0)|=1$, satisfy the growth condition
$$
\log |f(z)|\le \frac{K|1-z|^r}{(1-|z|)^p}, \qquad z\in\bd,
$$
where $0<p<r+1$.
Then for each $\tau>0$ there is a positive number $C_7=C_7(p,r,\tau)$ such that
\begin{equation}\label{btc111}
\sum_{\z\in Z(f),\,\frac{1-|\z|}{|1-\z|}>|1-\z|^{\beta}}
(1-|\z|)+\!\!
\sum_{\z\in Z(f),\,\frac{1-|\z|}{|1-\z|}\le |1-\z|^{\beta}}
\frac{(1-|\z|)^{p+1+\tau}}{|1-\z|^{\min(p,r)+1+\tau}}\le
C_7\cdot K,
\end{equation}
where
$$ \beta=\left\{
    \begin{array}{ll}
     1/(p+\tau) , & p\le r; \\
     (r+1-p)/(p+\tau), & r<p<r+1.
    \end{array}
  \right.
  $$
\end{theorem}

\begin{corollary}\label{genbtc}
Under the same conditions,
$$
\sum_{\z\in Z(f)}
\frac{(1-|\z|)^{p+1+\tau}}{|1-\z|^{\min(p,r)+\tau}}\le
C_8\cdot K.
$$
\end{corollary}

\begin{proof} We use that $(1-|\z|)/(|1-\z|)\le 1$ and that $1-|\z|\le 1$.
Then
$$
\sum_{\z\in Z(f),\,\frac{1-|\z|}{|1-\z|}>|1-\z|^{\beta}}
\frac{(1-|\z|)^{p+1+\tau}}{|1-\z|^{\min(p,r)+\tau}}\le \sum_{\z\in Z(f),\,\frac{1-|\z|}{|1-\z|}>|1-\z|^{\beta}}
(1-|\z|).
$$
It remains to use \eqref{btc111}.
\end{proof}

\begin{proof}[Proof of Theorem \ref{th031}] Let $\psi(z)=\frac{1-z}{1+z}$, $F(z)=f(\psi(z))$. Then $F$ is analytic in the right half-plane $\bc_r$ and
\begin{equation}\label{do1}
\log |f(z)|\le C' K\cdot \frac{|z|^r}{x^p}, \qquad z\in \bc_r,\,|z|<C,
\end{equation}
where $z=x+iy$, $C>1$ is arbitrary and $C'$ depends on $p,r,C$.

Let $\lambda >1$ be fixed later on. Consider the domain
$$
\Omega_0=\{x+iy:x>|y|^\lambda \}.
$$
Let $\phi_0$ be a conformal map of $\Omega_0$ onto $\bc_r$ such that $\phi_0(\Omega_0\cap\mathbb R)=\bc_r\cap \mathbb R$.
To obtain good asymptotic information on $\phi_0$ at $0$, we use the results of Warschawski \cite{wa} (see also  \cite{po}, Theorem~11.16).
For some $C$ and $C'$ depending only on $\lambda $ we obtain
\begin{equation}\label{do2}
0<C<\frac{|\phi_0(z)|}{|z|}<C'<\infty, \qquad z\in \Omega_0,\,|z|<1.
\end{equation}
Furthermore, the same results show that given $\gamma\in(0,1]$ we have
\begin{equation}\label{do3}
0<C<\frac{\re \phi_0(x+iy)}{x}<C'<\infty, \qquad x>\max(\gamma|y|,2|y|^\lambda ),\,|x+iy|<1,
\end{equation}
with $C, C'$ depending only on $\gamma$ and $\lambda $.

Next, we need a similar estimate for
\begin{equation}\label{do4}
2|y|^\lambda \le x<\gamma|y|,\quad |x+iy|<\frac14.
\end{equation}
For $y\in[-1/4,1/4]$ we consider the points
\begin{align*}
&A=|y|^\lambda /2+i(y-\gamma|y|),\quad &B&=|y|^\lambda /2+i(y+\gamma|y|),\\
&A'=3|y|^\lambda /2+i(y-\gamma|y|),\quad &B'&=3|y|^\lambda /2+i(y+\gamma|y|),\\
&A''=2|y|+i(y-\gamma|y|),\quad &B''&=2|y|+i(y+\gamma|y|)
\end{align*}
and the rectangles $ABB''A''$, $A'B'B''A''$, see Figure \ref{f04}.
From now on we fix $\gamma=\gamma(\lambda )$ as the maximal number in $(0,1]$ such that
$$
A'B'B''A''\subset ABB''A''\cap \overline{\Omega_0}
$$
for all $y\in[-1/4,1/4]$.

\begin{figure}[htbp]
\includegraphics[width=9cm]{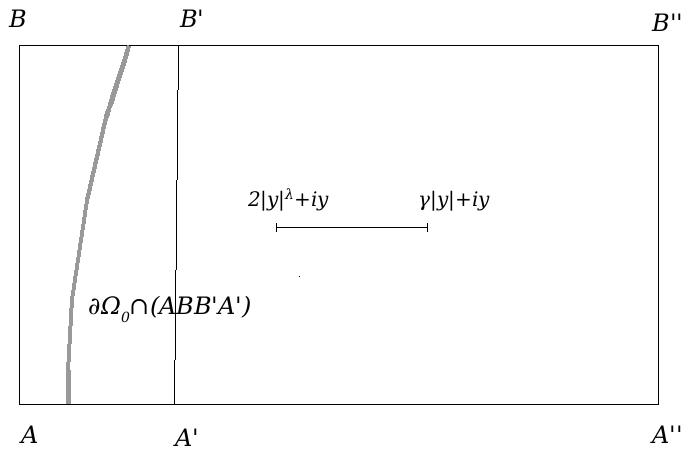}
\caption{Rectangles $ABB''A''$, $A'B'B''A''$, and the part of the boundary $\pt\Omega_0\cap (ABB'A')$.}
\label{f04}
\end{figure}

Fix $x+iy$ satisfying \eqref{do4}, the corresponding points $A,B,A',B',A'',B''$ and the rectangles $ABB''A''$, $A'B'B''A''$.
Then
$$
[A'',B'']\subset \{x+iy:x>\max(\gamma|y|,2|y|^\lambda ),\,|x+iy|<1\}.
$$

Set $u=\re \phi_0$. By \eqref{do2} and \eqref{do3} we have
\begin{gather*}
0\le u(w)\le C''|y|,\qquad w\in ABB''A''\cap \Omega_0,\\
C''|y|\le u(w),\qquad w\in [A'',B''],
\end{gather*}
with $C''$'s depending only on $\lambda $.

Since $0<x<|y|$, $\dist(x+iy,[A',B'])\ge x/4$, $\dist(x+iy,[A',A''])=\dist(x+iy,[B',B''])=\gamma|y|$, $\dist(x+iy,[A'',B''])\le 2|y|$,
an elementary estimate of harmonic measure shows that
$$
\omega(x+iy,[A'',B''],A'B'B''A'')\ge C''\cdot\frac{x}{|y|},
$$
with $C''$ depending only on $\lambda $.
Hence,
$$
u(x+iy)\ge C\cdot x.
$$
Since $0<x<|y|$, $\dist(x+iy,[A,B])\le x$, $\dist(x+iy,[A,A''])=\dist(x+iy,[B,B''])=\gamma|y|$, $\dist(x+iy,[A'',B''])\ge |y|$,
another elementary estimate of harmonic measure gives that
\begin{multline*}
\omega(x+iy,\partial(ABB''A''\cap \Omega_0)\setminus \partial\Omega_0, ABB''A''\cap \Omega_0)\\ \le
\omega(x+iy,\partial(ABB''A'')\setminus [A,B],ABB''A'')\le C''\cdot\frac{x}{|y|},
\end{multline*}
with $C''$ depending only on $\lambda $.
Hence,
$$
u(x+iy)\le C\cdot x.
$$
As a result, we obtain
\begin{equation}\label{do5}
0<C<\frac{\re \phi_0(x+iy)}{x}<C'<\infty, \qquad x\ge 2|y|^\lambda ,\,|x+iy|<1,
\end{equation}
with $C,C'$ depending only on $\lambda $.

Now, for $n\ge 1$ we define
$$
\Omega_n=\{x+iy:x>2^{-n}|y|^\lambda \},
$$
and $\phi_n:\Omega_n\mapsto \bc_r$,
$$
\phi_n(z)=2^{n/(\lambda -1)}\phi_0(2^{-n/(\lambda -1)}z).
$$
By \eqref{do2} and \eqref{do5}, for some $C,C'$ and for $n\ge 1$, $z\in \Omega_n$, $|z|<1$ we have
\begin{gather*}
0<C<\frac{|\phi_n(z)|}{|z|}<C'<\infty, \\
0<C<\frac{\re \phi_n(x+iy)}{x}<C'<\infty.
\end{gather*}

Next, we define
$$
G_n=F\circ \phi_n^{-1},\qquad n\ge 1.
$$
Then $|G_n(\phi_n(1))|=1$. Set $Q=\max(2|\phi_n(1)|, 1)$.
By \eqref{do1} we have
\begin{align*}
\log|G_n(iy)|&\le C K\cdot 2^{np}|y|^{r-\lambda p},\qquad y\in[-Q,Q],\\
\log|G_n(e^{i\theta}Q)|&\le C K\cdot 2^{np},\qquad \theta\in[-\pi/2,\pi/2],
\end{align*}
with $C$ depending only on $\lambda $.

From now on we suppose that $r-\lambda p>-1$. By the Poisson--Jensen formula in the right half-disk $\{z\in\mathbb C_r:|z|<Q\}$ we obtain that
$$
\sum_{G_n(x+iy)=0,\, |x+iy|<\frac12}x\le C K\cdot 2^{np}, \qquad n\ge 1,
$$
and hence,
$$
\sum_{F(x+iy)=0,\,x>2^{1-n}|y|^\lambda ,\, |x+iy|< C}x\le C' K\cdot 2^{np}, \qquad n\ge 1,
$$
with $C$ and  $C'$ depending only on $\lambda ,p,r$.
Theorem~\ref{th03} implies that
$$
\sum_{F(x+iy)=0,\,x>2^{1-n}|y|^\lambda ,\, C \le |x+iy|<1}x\le C' K\cdot 2^{np}, \qquad n\ge 1,
$$
with $C$ and $C'$ depending only on $\lambda ,p,r$.
Hence,
$$
\sum_{F(x+iy)=0,\,x>2^{1-n}|y|^\lambda ,\, |x+iy|<1}x\le C K\cdot 2^{np}, \qquad n\ge 1,
$$
with $C$ depending only on $\lambda ,p,r$.

Let $\delta>1$. Then
$$
\sum_{F(x+iy)=0,\,x>|y|^\lambda ,\, |x+iy|<1}x+
\sum_{F(x+iy)=0,\,x\le |y|^\lambda ,\, |x+iy|<1}\frac{x^{1+\delta p}}{|x+iy|^{\delta \lambda p}}
\le C K,
$$
with $C$ depending only on $\lambda ,p,r,\delta$.
If $p\le r$, then, given $\tau>0$, we can choose $\delta=1+\tau/p$, $\lambda =1+1/(p+\tau)$ to get
$$
\sum_{F(x+iy)=0,\,x>|y|^\lambda ,\, |x+iy|<1}x+
\sum_{F(x+iy)=0,\,x\le |y|^\lambda ,\, |x+iy|<1}\frac{x^{p+1+\tau}}{|x+iy|^{p+1+\tau}}
\le C K.
$$
If $r<p<r+1$, then, given $\tau>0$, we can choose $\delta=1+\tau/p$, $\lambda =(r+1+\tau)/(p+\tau)$ to get
$$
\sum_{F(x+iy)=0,\,x>|y|^\lambda ,\, |x+iy|<1}x+
\sum_{F(x+iy)=0,\,x\le |y|^\lambda ,\, |x+iy|<1} \frac{x^{p+1+\tau}}{|x+iy|^{r+1+\tau}}
\le C K.
$$
Returning to the zeros of $f$ and estimating those far from the point $1$ as in the proof of Theorem~\ref{th03} we obtain \eqref{btc111}.
\end{proof}

\subsection{Upper half-plane and plane with a cut}

A version of Theorem \ref{th02} for the upper half-plane looks as follows. We use a convenient shortening
\begin{equation*}
\{u\}_{c,\ep}:=(u_- -1+\ep)_+ -\min(c,u_+), \qquad c\ge0, \quad \ep>0, \quad u=u_+-u_-\in\br.
\end{equation*}

\begin{theorem}\label{upper}
Let $X=\{x_j\}_{j=1}^n$ and $X'=\{x'_k\}_{k=1}^m$ be two disjoint finite sets of distinct points
on the real line. Let $g\in \ca(\bc_+)$, $|g(i)|=1$, satisfy the growth condition
\begin{equation}\label{btc71}
\log|g(w)|\le
K\frac{(1+|w|)^{2b}}{(\im w)^a}\,\frac{\prod_{j=1}^n|w-x_j|^{c_j}}{\prod_{k=1}^m |w-x'_k|^{d_k}}, \quad w\in\bc_+,
\end{equation}
and $a,b,c_j,d_k\ge 0$. Denote
$$ l:=2a-2b-\sum_{j=1}^n c_j+\sum_{k=1}^m d_k=l_+ - l_-. $$
Then for each $\ep>0$ there exists a positive number $C_9=C_9(X,X',a,b,c_j,d_k,\ep)$ such that
the following Blaschke-type condition holds:
\begin{equation}\label{btc7}
\sum_{\zeta\in Z(g)} \frac{(\im \zeta)^{a+1+\ep}}{(1+|\zeta|)^{l_1}}\,
\frac{\prod_{k=1}^m|\zeta-x'_k|^{(d_k-1+\ep)_+}}{\prod_{j=1}^n |\zeta-x_j|^{\min(a,c_j)}}\le
C_9 \cdot K,
\end{equation}
where the parameter $l_1$ is defined by the relation
\begin{equation*}
l_1 :=2(a+1+\ep)+\{l\}_{a,\ep} -\sum_{j=1}^n \min(a,c_j)+\sum_{k=1}^m (d_k-1+\ep)_+.
\end{equation*}
\end{theorem}

\begin{proof} Since the result follows directly from Theorem \ref{th02}, we give only a sketch of the proof.
Consider the standard conformal  mappings
\begin{equation}\label{confma}
z=z(w)=\frac{w-i}{w+i} :\bc_+\to \bd, \quad w=w(z)=i\, \frac{1+z}{1-z}:\bd\to\bc_+,
\end{equation}
and the following elementary relations between the corresponding
quantities in $\bc_+$ and $\bd$:
$$
\frac2{1+|w|}\le |1-z|\le \frac{2\sqrt{2}}{1+|w|}\,, \qquad \frac{2\,\im w}{(1+|w|)^2}
\le 1-|z|\le\frac{8\,\im w}{(1+|w|)^2}.
$$
We have
$$ |w-x_j|=\frac{2|z-\z_j|}{|1-z||1-\z_j|}\,, \quad
\frac{2|w-x_j|}{(1+|w|)|x_j+i|}\le |z-\z_j|\le\frac{2\sqrt{2}\,|w-x_j|}{(1+|w|)|x_j+i|} $$
with $\z_j=z(x_j)$. Similar inequalities hold for $|w-x'_k|$ and $|z-z(x'_k)|$. Then, we map $\bc_+$ onto
$\bd$ using $w(z)$ defined in \eqref{confma}, and rewrite inequality \eqref{btc71} in terms of $z\in \bd$.
To complete, we apply Theorem \ref{th02} and go back to $\bc_+$ using $z(w)$ defined in \eqref{confma}.
\end{proof}

In view of applications to the spectral theory we give yet another version of Theorem \ref{th02}
related to the domain $\bc\backslash\br_+$.

\begin{theorem}\label{cut}
Let $T=\{t_j\}_{j=1}^n$ and $T'=\{t'_k\}_{k=1}^m$ be two disjoint
finite sets of distinct positive numbers. Let $h\in \ca(\bc\backslash\br_+)$, $|h(-1)|=1$,
satisfy the growth condition
\begin{equation*}
\log|h(\l)|\le \frac{K}{|\l|^r}\,\frac{(1+|\l|)^{b}}{\dist^a(\l,\br_+)}\,\frac{\prod_{j=1}^n |\l-t_j|^{c_j}}
{\prod_{k=1}^m |\l-t'_k|^{d_k}}, \quad \l\in\bc\bsl\br_+,
\end{equation*}
and  $a,b,c_j,d_k\ge0$, $r\in\br$. Denote
$$ s:=3a-2b+2r-2\sum_{j=1}^n c_j+2\sum_{k=1}^m d_k=s_+ - s_-. $$
Then for each $\ep>0$ there is a positive number $C$ which depends on all parameters
involved such that the following inequality holds:
\begin{equation}\label{btc08}
\sum_{\z\in Z(h)} \dist^{a+1+\ep}(\z,\br_+)\,\frac{|\z|^{s_1}}{(1+|\z|)^{s_2}}\cdot
\frac{\prod_{k=1}^m |\z-t'_k|^{(d_k-1+\ep)_+}}{\prod_{j=1}^n |\z-t_j|^{\min(a,c_j)}}
\le C\cdot K,
\end{equation}
where the parameter $s_1$, $s_2$ are defined by the relations
\begin{equation*}
\begin{split}
s_1 &:= \frac{\{-2r-a\}_{a,\ep} -a-1-\ep}2\,, \\
s_2 &:= a+1+\ep + \frac{\{-2r-a\}_{a,\ep}+\{s\}_{a,\ep}}2 -\sum_{j=1}^n \min(a,c_j)+\sum_{k=1}^m (d_k-1+\ep)_+ \,.
\end{split}
\end{equation*}

\end{theorem}

The result is a direct consequence of Theorem \ref{upper} applied to the function $g(w):=h(w^2)$, $w\in\bc_+$,
and the elementary inequalities
$$ |w|\,\im w\le \dist(w^2,\br_+)\le 2|w|\,\im w, \qquad w\in\bc_+. $$

\if{
\section{Analytic Fredholm alternative}\label{s17}

The goal of this section is to refine partially (for a certain range of the parameters) a
recent result of R. Frank \cite[Theorem 3.1]{fra15} on some quantitative aspects of the analytic Fredholm alternative.
Precisely, the problem concerns the distribution of eigenvalues of finite type of an
operator-valued function $W(\cdot)=I+T(\cdot)$, analytic on a domain $\oo$ of the complex plane.
We always assume that $T\in\Sc_\infty$, the set of compact operators on the Hilbert space.
A  number $\l_0\in\oo$ is called an {\it eigenvalue of finite type of} $W$ if $\ker W(\l_0)\not=\{0\}$,
(that is, $-1$ is an eigenvalue of $T(\l_0)$), if $W(\l_0)$ is Fredholm (that is, both
$\dim\ker W(\l_0)$ and ${\rm codim\ ran}\, W(\l_0)$ are finite), and if $W$ is invertible in some
punctured neighborhood of $\l_0$. As is known, see, e.g., \cite[Theorem XI.8.1]{ggk}, the function $W$ admits the following
expansion at any eigenvalue of finite type
$$ W(\l)=E(\l)(P_0+(\l-\l_0)^{k_1}\,P_1+\ldots+(\l-\l_0)^{k_l}\,P_l)G(\l), $$
where $P_1,\ldots,P_l$ are mutually disjoint projections of rank one, $P_0=I-P_1-\ldots-P_l$,
$k_1\le\ldots\le k_l$ are positive integers, and $E$, $G$ are analytic operator-valued functions,
defined and invertible in some neighborhood of $\l_0$. The number
$$ \nu(\l_0,W):=k_1+\ldots+k_l $$
is usually referred to as an {\it algebraic multiplicity of the eigenvalue} $\l_0$.

The following result (a part of Theorem 3.1), is a cornerstone of the paper \cite{fra15}.

{\bf Theorem F}. Let $T(\cdot)$ be an analytic operator-valued function on the domain
$\oo=\bc\backslash\br_+$, so that $T\in\Sc_p$, $p\ge 1$, the set of the Schatten--von Neumann operators
of order $p$. Assume that for all $\l\in\bc\backslash\br_+$
\begin{equation}\label{eq1}
\|T(\l)\|_p\le \frac{M}{\dist^\rho(\l,\br_+)\,|\l|^\s}\,, \qquad  \rho>0, \ \ \s\in\br, \ \ \rho+\s>0.
\end{equation}
Let $\{\l_j\}$ be the eigenvalues of $W$ of finite type, repeated according to their algebraic multiplicity.
Denote $q:=(p\rho+2p\s-1+\ep)_+$. Then for all $\ep>0$
\begin{equation}\label{fr}
\sum_{|\l_j|\le M^{1/\rho+\s}}\dist^{p\rho+1+\ep}(\l_j,\br_+)\, |\l_j|^{\frac{q-p\rho-1-\ep}2} \le
C\,M^{\frac{q+p\rho+1+\ep}{2(\rho+\s)}}\,.
\end{equation}
Here $C$ is a generic positive constant which depends on $p, \rho, \s, \ep$.

The similar results for the eigenvalues outside the disk of an appropriate radius and for
$\rho=0$ are also available.

The proof of this result is based on the identification of the eigenvalues of finite type of $W$ with
the zeros of certain {\it scalar} analytic functions, known as the regularized determinants
$$ f(\l):={\rm det}_{p}(I+T(\l)), $$
see, e.g., \cite{Si05} and \cite{gokr-nsa} for their definition and basic properties. The point is that the set of eigenvalues
of finite type of $W$ agrees with the zero set of $f$, and moreover, $\nu(\l_0,W)=\mu_f(\l_0)$, the multiplicity of zero of
$f$ at $\l_0$ (see \cite[Lemma 3.2]{fra15} for the rigorous proof). Thereby, the problem is reduced to the study of the zero
distributions of certain analytic functions, the latter being a classical topic of complex analysis going back to Jensen and Blaschke.

A key ingredient of the author's proof is \cite[Theorem 0.2]{bgk1}. We are aimed at proving the result, which refines Theorem F
for certain range of the parameters, by using Theorem \ref{cut} instead.

\begin{theorem}
Let $T(\cdot)$ be an analytic operator-valued function on the domain $\oo=\bc\backslash\br_+$,
which satisfies the hypothesis of Theorem F. Assume also that
\begin{equation}\label{range}
-\rho<\s\le-\frac{\rho}2\,.
\end{equation}
Let $\{\l_j\}$ be the eigenvalues of $W$ of finite type, repeated according to their algebraic multiplicity.
Then for all $0<\ep<1$
\begin{equation}\label{bgk}
\sum_{|\l_j|\le M^{1/\rho+\s}}  \dist^{p\rho+1+\ep}(\l_j,\br_+)\, |\l_j|^{p\s-\frac{1+\ep}2}\le
C M^{p\s+p\rho+\frac{1+\ep}2}.
\end{equation}
\end{theorem}

\begin{proof}
We follow the line of reasoning from \cite{fra15}. The scaling $T_1(\l):=T(M^{1/\rho+\s}\,\l)$ looks reasonable, so
$$ \|T_1(\l)\|_p\le \frac{1}{\dist^\rho(\l,\br_+)\,|\l|^\s}\,, $$
and, by \cite[Theorem 9.2, (b)]{Si05}, we have for the determinant $f_1=\det_{p}(I+T_1)$
\begin{equation}\label{eq2}
\log|f_1(\l)|\le \frac{\gga_p}{\dist^{p\rho}(\l,\br_+)\,|\l|^{p\s}}\,, \qquad \l\in\bc\backslash\br_+.
\end{equation}

To apply Theorem \ref{cut} we have to ensure the normalization condition. Note that due to the assumptions
on the parameters in \eqref{eq1}
the function $T_1$ tends to zero along the left semi-axis, so the inequality (see \cite[Theorem 9.2, (c)]{Si05})
$$ |f_1(\l)-1|\le \p(\|T_1(\l)\|_p), \quad \p(t):=t\exp\bigl(\gga_p(t+1)^p\bigr), \quad t\ge0 $$
provides a lower bound for $f_1$ as long as the right side is small enough. We have for $t\ge1$ and $\l=-t\in\br_-$
$$ |f_1(-t)-1|\le \frac{C_1}{t^{\rho+\s}}\,, $$
(in the sequel $C_k$ stand for generic positive constants depending on the parameters involved). If $t\ge (2C_1)^{1/\rho+\s}=C_2$, then
$|f_1(-t)|\ge 1/2$, and so
\begin{equation}\label{eq3}
\log|f_1(-t)|\ge -2(1-|f_1(-t)|)\ge -\frac{2C_1}{t^{\rho+\s}}\,.
\end{equation}

Next, put
$$ h(\l):=\frac{f_1(t\l)}{f_1(-t)}\,, \qquad h(-1)=1. $$
It follows now from \eqref{eq2} and \eqref{eq3} that for $t\ge C_2$
\begin{equation*}
\begin{split}
\log|h(\l)| &= \log|f_1(t\l)|-\log|f_1(-t)|\le
\frac{\gga_p}{t^{\rho+\s}}\,\frac{1}{\dist^{p\rho}(\l,\br_+)\,|\l|^{p\s}}+\frac{2C_1}{t^{\rho+\s}} \\
&\le \frac{C_3}{t^{\rho+\s}}\,\lp\frac{1}{\dist^{p\rho}(\l,\br_+)\,|\l|^{p\s}}+1\rp
\le  \frac{C_3}{t^{\rho+\s}}\,\frac{(1+|\l|)^{p(\rho+\s)}}{\dist^{p\rho}(\l,\br_+)\,|\l|^{p\s}}\,.
\end{split}
\end{equation*}
Theorem \ref{cut} applies with
$$ a=p\rho, \quad r=p\s, \quad b=p(\rho+\s), \quad K=\frac{C_3}{t^{\rho+\s}}\,, $$
so $s=a$, $\{s\}_{a,\ep}=-a$, and in view of \eqref{range}
$$ \{-2r-a\}_{a,\ep}=-\min(a, -2r-a)=2r+a=p\rho+2p\s, $$
(recall that by the assumption $a>-2r-a$). Hence
$$ s_1=\frac{2p\s-1-\ep}2\,, \qquad s_2=p\rho+p\s+1+\ep, $$
and \eqref{btc08} implies
\begin{equation*}
\sum_{z\in Z(h)} \dist^{p\rho+1+\ep}(z,\br_+)\,\frac{|z|^{\frac{2p\s-1-\ep}2}}{(1+|z|)^{p\rho+p\s+1+\ep}}
\le \frac{C_4}{t^{\rho+\s}}\,,
\end{equation*}
or
\begin{equation*}
t^{\frac{1+\ep}2}\,\sum_{\z\in Z(f_1)} \dist^{p\rho+1+\ep}(\z,\br_+)\,\frac{|\z|^{\frac{2p\s-1-\ep}2}}{(t+|\z|)^{p\rho+p\s+1+\ep}}
\le \frac{C_4}{t^{\rho+\s}}\,.
\end{equation*}

For $|\z|\le1$ we fix $t$, say, $t=C_2$, and since $t+|\z|\le C_2+1$, we come to
\begin{equation*}
\sum_{\z\in Z(f_1)\cap\bar\bd} \dist^{p\rho+1+\ep}(\z,\br_+)\,|\z|^{p\s-\frac{1+\ep}2}\le C_5,
\end{equation*}
which, after scaling, is \eqref{bgk}.
\end{proof}

Note that under assumption \eqref{range}
$$ p\s-\frac{1+\ep}2\le -\frac{p\rho+1+\ep}2<0, $$
so for $|\z|\le1$
$$ |\z|^{p\s-\frac{1+\ep}2}\ge |\z|^{-\frac{p\rho+1+\ep}2}\,, $$
that is, \eqref{bgk} is stronger than \eqref{fr} with regard to eigenvalues tending to zero.
As far as the rest of the values of $\rho$ and $\s$ and the eigenvalues tending to infinity go, Theorem \ref{cut}
gives the same results as in \cite[Theorem 3.1]{fra15}.

\begin{remark}
The similar results can be obtained under (formally) more general assumptions
\begin{equation*}
\|T(\l)\|_p\le M\frac{(1+|\l|)^\tau}{\dist^\rho(\l,\br_+)\,|\l|^\s}\,, \qquad  \rho>0, \ \ \s\in\br, \ \ \rho+\s>\tau\ge0.
\end{equation*}
\end{remark}
}\fi

\medskip
\nt {\bf Acknowledgments.}
Recently, Peter Yuditskii has celebrated his 60-th anniversary. All three authors are Peter's friends, LG is his former associate at ILTPE in Kharkov,
and SK is his former Ph.D. student.  We would like to take this opportunity to warmly wish Peter many more years of flourishing scientific
(and musical) activity as well as deep mathematical insights.

\smallskip
We thank D. Tulyakov for useful discussions. A part of this research was done during L. Golinskii's visit to University of Bordeaux in the framework
of ``IdEx Invited Scholar Programm''. He gratefully acknowledges the financial support of the IdEx as well as the hospitality of the institution.

\end{document}